\documentclass{article}
\usepackage[utf8]{inputenc}
\usepackage[linesnumbered,ruled,vlined,titlenotnumbered]{algorithm2e}
\usepackage{amsthm}

\newtheorem{theorem}{Theorem}[section]
\newtheorem{corollary}{Corollary}[section]
\newtheorem{lemma}[theorem]{Lemma}

\providecommand{\keywords}[1]{\textbf{\textit{Keywords---}} #1}

\title{A Linear Algorithm for Computing Independence Polynomials of Trees}

\author{
        Ohr Kadrawi \\
        Department of Mathematics\\
        Ariel University\\
        Ariel 4070000, \underline{Israel}\\
        orka@ariel.ac.il
        \and
        Vadim E. Levit \\
        Department of Mathematics\\
        Ariel University\\
        Ariel 4070000, \underline{Israel}\\
        levitv@ariel.ac.il
        \and
        Ron Yosef \\
        Department of Computer Science\\
        The Hebrew University of Jerusalem\\
        Jerusalem 9190401, \underline{Israel}\\
        ron.yosef@mail.huji.ac.il\\
        \and
        Matan Mizrachi \\
        Department of Computer Science\\
        HIT-Holon Institute of Technology\\
        Holon, Golomb 52, \underline{Israel}\\
        matanmi@my.hit.ac.il
        }

\begin{document}

\maketitle

\begin{abstract}
    An independent set in a graph is a set of pairwise non-adjacent vertices. Let $\alpha(G)$ denote the cardinality of a maximum independent set in the graph $G = (V, E)$. Gutman and Harary defined the independence polynomial of $G$
    \[
    I(G;x) = \sum_{k=0}^{\alpha(G)}{s_k}x^{k}={s_0}+{s_1}x+{s_2}x^{2}+...+{s_{\alpha(G)}}x^{\alpha(G)},
    \]
    where $s_k$ denotes the number of independent sets of cardinality $k$ in the graph $G$ \cite{GH}. A comprehensive survey on the subject is due to Levit and Mandrescu \cite{LM}, where some recursive formulas are allowing to calculate the independence polynomial. A direct implementation of these recursions does not bring about an efficient algorithm. Yosef, Mizrachi, and Kadrawi developed an efficient way for computing the independence polynomials of trees with $n$ vertices, such that a database containing all of the independence polynomials of all the trees with up to $n-1$ vertices is required \cite{YMK}. This approach is not suitable for big trees, as an extensive database is needed. On the other hand, using dynamic programming, it is possible to develop an efficient algorithm that prevents repeated calculations. In summary, our dynamic programming algorithm runs over a tree in linear time and does not depend on a database.
\end{abstract}

\section{Introduction}
\subsection{Definitions}
Throughout this paper $G = (V, E)$ is a simple (i.e., a finite, connected, undirected, loopless and without multiple edges) graph with vertex set $V = V(G)$ of cardinality $|V(G)| = n(G)$ and edge set $E = E(G)$ of cardinality $|E(G)| = m(G)$.\\

A set $S \in V(G)$ is \textit{independent} if no two vertices in $S$ are adjacent. An independent set of maximum cardinality, named a \textit{maximum independent set}, and its size is called the \textit{independence number} and denoted $\alpha(G)$. An \textit{independence polynomial} of graph $G$ defined by Gutman and Harary in 1983 \cite{GH}:
\begin{equation}
    \label{eq:independent-polynomial}
    I(G;x) = \sum_{k=0}^{\alpha(G)}{s_k}x^{k}={s_0}+{s_1}x+{s_2}x^{2}+...+{s_{\alpha(G)}}x^{\alpha(G)}
\end{equation}

where $s_k$ denotes the number of independent sets of cardinality $k$ in the graph $G$. $s_0=1$ is the number of independent sets with minimal cardinality in $G$ (i.e empty set) \cite{LMWCT}.\\

A \textit{tree decomposition} of a general graph $G(V, E)$ is a tree $T$ of "bags", where if edge $(u, v)\in E(G) $ then $u$ and $v$ are together in same "bag", and $\forall u \in V(G)$ the “bags” containing $u$ are connected in $T$. The \textit{width} of a tree decomposition equal to one less than the maximum bag size and the \textit{treewidth} of $G(V, E)$ equal to least width of all tree decompositions for $G$.\\

\subsection{Computing the independence polynomial}

To compute the \emph{independence polynomials} $I(G; x)$, as shown in Levit and Mandrescu survey \cite{LM} one can use the following recursive formula:
        \begin{equation}
        \label{eq:independent-polynomial-rec}
            I(G; x) = I(G - v;x) + x \cdot I(G-N[v]; x)
        \end{equation}
        where N[v] is the neighborhood of the vertex v including v itself. \\

To compute the \emph{independence polynomials} of union of graphs the formula is \cite{CX, LM}:
        \begin{equation}
        \label{eq:independent-polynomial-graphs}
            I(G_1 \cup G_2; x) = I(G_1;x) \cdot I(G_2;x)
        \end{equation}

\subsection{Problem}
Some facts on similar problems:
\begin{itemize}
    \item Computing the \textbf{Maximum Independent Set} of a general graph is an NP-Complete problem.
    \item Computing the \textbf{Independence Polynomial} of a general graph is also an NP-Complete problem. The degree of the independence polynomial is equivalent to the Maximum Independent Set.
    \item There are families of graphs with a closed-form expression for their independence polynomial \cite{Fer}, but from what is known, trees are not one of them.
    \item Some families of graphs have a polynomial algorithm to compute their $\alpha(G)$. Does it give us hope for a polynomial algorithm to compute their independence polynomials?
\end{itemize}

According to Bodlaender \cite{Bod}, many practical problems rely heavily on graphs with bounded treewidth.
Tittmann \cite{Titt} offers a way for the construction of an algorithm that computes the independence polynomial of a graph with bounded treewidth in polynomial time based on a specific partition.\\
Based on Tittmann, who described in acronyms how an algorithm for graphs with a bounded treewidth should look, we accurately investigated an actual algorithm for trees.\\

Yosef, Mizrachi, and K. developed an efficient approach for computing the independence polynomials of trees with $n$ vertices \cite{YMK}. This approach is based on an extensive database containing the independence polynomials of all the trees with up to $n-1$ vertices, and the induction step computing the independence polynomials of all the trees with $n$ vertices based on their $n-1$ counterparts. Although this approach works well with small-sized trees, big trees may require an enormous database to compute their independence polynomials.

\section{Main idea of algorithm}
    The new algorithm is not based on a database nor requires one. Instead, it used the dynamic programming
    technique.\\

    The independence polynomial $I(G; x)$ represented in the algorithm as a list with $\alpha(G) + 1$ cells in the following format:
        \[[s_{k=\alpha(G)}, s_{k=\alpha(G)-1}, ... , s_{k=1}, s_{k=0}]\]

    \begin{table}[ht!]
    \centering
        \begin{tabular}{||c c c||}
            \hline
            Graph & $I(G;x)$ & Stored as \\ [0.5ex]
            \hline\hline
            $P_1$ & $x+1$ & $[1, 1]$ \\
            $P_2$ & $2x+1$ & $[2, 1]$ \\
            $P_3$ & $x^2+3x+1$ & $[1, 3, 1]$ \\[1ex]
            \hline
        \end{tabular}
        \caption{Path graphs with 1,2 and 3 vertices and their representation in the algorithm.}
        \label{ExmpleOfStore}
    \end{table}

    From the computing process, the independence polynomial formula, as described in (\ref{eq:independent-polynomial-rec}), says that every vertex can be calculated just after its children and grandchildren are calculated. Post-order-traversal validates that the children and grandchildren vertices will be computed before the father vertex.\\

    To examine all tree graphs, we divided them into three cases:\\
    \begin{enumerate}
        \item \textbf{$|V| = 1$}:
        \begin{itemize}
            \item  $I(G - v;x) = [1]$ because $G-v$ has no vertices in $G-v$, there is one independent set of cardinality zero (i.e., empty set).
            \item $I(G-N[v]; x) = [1]$ from the same reason above. so:
        \end{itemize}
        \begin{center}
            $I(G; x) = [1] + [1, 0] \cdot [1] = [1, 1]$
        \end{center}

        \item \textbf{$|V| = 2$}: \begin{itemize}
            \item  $I(G - v;x) = [1, 1]$ because when vertex $v$ removed, the graph stay with only one vertex, and this sub-graph matches the previous case.
            \item  $I(G-N[v]; x) = [1]$ because $G-N[v]$ has no vertices in $G-v$, there is one independent set of cardinality zero (i.e., empty set). so:
        \end{itemize}
        \begin{center}
            $I(G; x) = [1, 1] + [1, 0] \cdot [1] = [2, 1]$
        \end{center}

        \item \textbf{$|V| > 2$}:
        \begin{enumerate}
        \item Start with traveling on the tree in post-order traversal. When reaching a leaf node, like case 1, it can calculate by A:
        \begin{itemize}
            \item $I(G - v;x) = [1]$
            \item $x \cdot I(G-N[v]; x) = [1, 0]$
        \end{itemize}
        \begin{center}
            $I(G; x) = [1, 1] $
        \end{center}
        \item For an inner vertex that all its children were calculated, when vertex v is removed, the number of connected components can rise, and in such case, computation of sub-graphs union, as described in  (\ref{eq:independent-polynomial-graphs}), is needed.\\

        So in purpose to calculate $I(T; x)$, compute next parameters:
        \begin{itemize}
            \item $I(T - v;x) = \Pi_{u \in children[v]} I(T; x)$
            \item $I(T-N[v]; x)$ parameter said that we remove the vertex v with its neighbors so we can use $I(T - v;x)$ parameter of the children:
            \begin{center}
                $I(T - N[v];x) = \Pi_{u \in children[v]} I(T-u; x)$
            \end{center}
            \end{itemize}
        \end{enumerate}
    \end{enumerate}

    Finally use formula \ref{eq:independent-polynomial-rec} to calculate $I(T; x)$.

\section{Main Algorithm}
    The algorithm in FIP (find independence polynomial) function starts at selected root node of a tree and goes as far as it can down a given branch in FIPR (find independence polynomial recursion) function, then backtracks until it finds an unexplored path, and then explores it. The algorithm does this until the entire graph has been explored.\\

    \begin{algorithm}[H]
        \caption{FIP(T)}
        \KwIn{Tree T as adjacency list}
        \KwOut{A list IP that represents the independence polynomial of T }

        \tcp{Two base cases:}
        \If{$|V| == 1$}{\Return{[1,1]}}
        \If{$|V| == 2$}{\Return{[2,1]}}

        \tcp{Select a root vertex that not a leaf:}
        $root \gets findInnerVertex()$\\

        \tcp{Call the recursive function:}
        FIPR(T, root)\\

        \tcp{Return the independence polynomial of $T$:}
        \Return {$I_of_T(root)$}
    \end{algorithm}

    \begin{algorithm}[H]
        \caption{FIPR(T, root)}
        \KwIn{Tree T as adjacency list, root vertex}
        \KwOut{Three lists: I(T; x), I(T-v; x), I(T-N[v]; x)}

        \tcp{Stop condition:}
        \uIf{root is a leaf}{
            I\_of\_T\_minus\_v(root) $\gets$ [1]\\
            I\_of\_T\_minus\_Nv(root) $\gets$ [1]\\
            I\_of\_T(root) $\gets$ [1, 1]\\
        }
        \Else{
            \tcp{Explore undiscovered vertices that are root-neighbors:}
            \ForEach{vertex u in N[root] that never explored}{
                    \tcp{Save the hierarchy:}
                    vertex\_children(root).append(u)\\
                    \tcp{Call again recursively with u:}
                    FIPR(T, u)
            }
        $left \gets [1]$ \tcp{Neutral organ to multiply}
        $right \gets [1,0]$ \tcp{Initialize as x}
        \tcp{Calculate the union of sub-graphs}
        \ForEach{u in vertex\_Children}{
            $left \gets left * I\_of\_T(u)$ \tcp{polynomial multiplication}
            $right \gets right * I\_of\_T\_minus\_v(u)$\\
        }
        \tcp{Set all lists in $root$ index}
        I\_of\_T\_minus\_v(root) $\gets left$\\
        I\_of\_T\_minus\_Nv(root) $\gets right$\\
        I\_of\_T(root) $\gets left + right$  \tcp{polynomial Addition}
        }
    \end{algorithm}

\section{Proof of correctness}

%https://cs.stackexchange.com/questions/7749/dfs-proof-of-correctness
\begin{lemma}
    FIPR(T, root) is called exactly once for each vertex in the graph
\end{lemma}

\begin{proof}
    Clearly FIPR(T, root) is called for a vertex u only if it not discovered. The moment it's called,  FIPR(T, root) cannot be called vertex v again. Furthermore, because T is connected component, and FIPR(T, root) uses post order travesal in the implementation, it's ensures that it will be called for every vertex in T.
\end{proof}

\begin{lemma}
    The body of the foreach loop that explore undiscovered vertices that are root-neighbors (lines 6-8) is executed exactly once for each edge (v,u) in the graph.
\end{lemma}

\begin{proof}
    FIPR(T, root) is called exactly once for each vertex root (Claim 1). And the body of the loop is executed once for all the unseen edges that connect to root.
\end{proof}

\begin{corollary}
    The algorithm can start from every vertex root such that\\$deg(v) \geq 2$ and get the same runtime.
\end{corollary}

\begin{proof}
    The order of walking on the tree is in post order travels. In this way, it go through each edge exactly twice so that is no matter which vertex v $(deg(v) \geq 2)$ it start from, the running time will remain the same.
\end{proof}

Therefore, the running time of the FIPR(T, root) algorithm is $O(n+m)$, and the tested graphs are trees, so it summarise to $O(n)$.\\

\section{Conclusion}
    In this article we have presented a linear algorithm that uses dynamic programming to find an independence polynomial of trees. The algorithm uses a post-order-travel over the graph. when it reaches the leaves it calculates them, and in the backtracking, when all children were calculated, there comes the father turn to compute.

\end{document}